\documentclass[11pt]{amsart}
\numberwithin{equation}{section}
\usepackage[latin1]{inputenc}
\usepackage{amsmath,amsthm,amssymb}
\usepackage{a4}
\usepackage{esint}
\usepackage{epsfig,color}                   
\usepackage{graphicx, picinpar}       
\usepackage{subfig}          

\newcommand{\vc}[1]{{\bf #1}}
\newcommand{\ii}{{\rm i}}
\newcommand{\e}[1]{{\rm e}^{#1}}

\newtheorem{example}{Example}
\newtheorem{remark}{Remark}
\newtheorem{theorem}{Theorem}
\newtheorem{corollary}{Corollary}
\newtheorem{lemma}{Lemma}

\begin{document}
\title[Steepest descent applied to scattering integrals]{Applying the numerical method of steepest descent on multivariate oscillatory integrals in scattering theory}
\author{Andreas Asheim}
\address{Departement Computerwetenschappen,\\ KU Leuven, 3001 Heverlee, Belgium}
\email{andreas.asheim@cs.kuleuven.be}
\thanks{Supported by the Norwegian Research Council's Magne S. Espedahl Fellowship.}
\date{\today}

\begin{abstract}
In this paper we demonstrate that the numerical method of steepest descent fails when applied in a straight forward fashion to the most commonly occurring highly oscillatory integrals in scattering theory. Through a polar change of variables, however, the integral can be brought on a form that can be solved efficiently using a mix of oscillatory integration techniques and classical quadrature. The approach is described in detail and demonstrated numerically on integration problems taken from applications.
\end{abstract}

\maketitle

\section{Introduction}
Efficient numerical approximation of oscillatory integrals is a challenging problem with a wide range of applications. Multivariate oscillatory integration is particularly difficult, with methods and theoretical results in practice being limited in applicability to certain classes of integrals. This work concerns such a class, namely integrals of the form
\[
\int_D f(\vc{x})\e{\ii\omega g(\vc{x})}{\rm d}\vc{x},
\]
where $g(\vc{x})$ behaves locally around a point $\vc{x}_0\in D\subset \mathbb{R}^n$ like $|\vc{x}-\vc{x}_0|^\alpha$, and $\omega$ may be large compared to $D$. This particular kind of integral appears frequently in literature, typically with $n=2,3$ and $\alpha=1$, e.g. in acoustics \cite{pierce1989acoustics,Bouwkamp:1954wt}, and scattering theory \cite{nedelec,voronovich1999wave,ChandlerWilde:2012kq}, and more specifically in high frequency BEM \cite{BRUNO:2007gt}.  The Greens function of the Helmholtz equation possesses this kind of special point, and this is clearly one reason for the ubiquitousness of these kinds of integrals; any computation of a wave field in space from a surface field, like in e.g. BEM, involves integrating the Green's function.

An extensive amount of research has been devoted to methods for efficient evaluation of univariate integrals, integrals of the form,
\begin{equation}\label{eq:HOI}
\int_a^bf(x)\e{\ii\omega g(x)}{\rm d}x,
\end{equation}
where $\omega$ is a parameter that may be large. Recently developed methods for such integrals include variants of the Filon-type methods \cite{Iserles:usingderivatives,Iserles:2004ke,Dominguez:2011td,Asheim:2008eb}, the numerical method of steepest descent \cite{Huybrechs:analyticcontinuation,Asheim:2010kc}, and Levin-type methods \cite{Levin:1996kp,Olver:2005jp}. For an oscillatory integral of the form \eqref{eq:HOI}, certain special points, namely endpoints, singularities, and stationary points, i.e., points where $g'(x)=0$, determine the asymptotic behaviour of the integral as $\omega\to\infty$. The above mentioned methods treat these asymptotically contributing points specifically such that terms in the asymptotic expansion of the error vanish, leading to methods with an error of size $\mathcal{O}(\omega^{-\beta})$, for some $\beta>0$. For example, a Filon-type approximation is obtained by interpolating $f(x)$ with a number of derivatives at all contributing points, and the approximation is constructed a liner combination of moments.  In this case the asymptotic order $\beta$ of the approximation depends on how many derivatives that are interpolated \cite{Iserles:usingderivatives}.

Multivariate oscillatory integrals have the general form
\[
\int_D f(\vc{x})\e{\ii\omega g(\vc{x})}{\rm d}\vc{x},
\]
where now $D\subset \mathbb{R}^n$, and $f,g:\mathbb{R}^n\to\mathbb{R}$. In this case we can think of the endpoints in the univariate case as corresponding to corners of $\partial D$ in $n$ dimensions. The notion of a stationary points naturally extends to higher dimensions, being a point where $\nabla g(\vc{x})=0$. In addition, a new sort of contributing point without a straight-forward parallel in the univariate case will appear, namely the resonance point, which is a point where $\nabla g(\vc{x})$ is perpendicular to the boundary. This, and other added complications, are reason why general-purpose methods for multivariate oscillatory integrals can be considered a hard problem. The Filon-type methods do generalise in a quite natural way to higher dimensions \cite{Iserles:multivariate_derivatives,Iserles:multivariate_stationary}, however, moments are not generally available in higher dimensions, and supplementary methods are needed to obtain these. Levin-type methods do not rely on moments, but no general procedure for handling cases with stationary points and resonance points is available for these methods \cite{Olver:2006iq}. The numerical method of steepest descent can be generalised to multivariate cases by regarding multivariate integrals as nested univariate integrals \cite{Huybrechs:multivariate}, and this approach can handle stationary points and resonance points to a certain degree. This is likely the most general method among the above mentioned.

In this paper it will however be demonstrated that whenever $g(\vc{x})$ behaves locally around a point $\vc{x}_0\in D$ like $|\vc{x}-\vc{x}_0|$, this induces problems for the method of steepest descent. As such, the only method which seem applicable is the so-called localised method of stationary phase due to Ganesh et al. \cite{Ganesh:2007ee}, which was developed specifically for use in a BEM solver. Here it is suggested that a polar change of variables is needed to cancel a singularity of the amplitude function $f(\vc{x})$ at $\vc{x}_0$, as is commonly seen in scattering integrals, and the remaining part is integrated with a combination of oscillatory and non-oscillatory techniques. In this work it will however become apparent that this change of variables is also precisely what is needed in order to apply the numerical method of steepest descent. 

This paper is built up as follows: Some relevant facts about the numerical steepest descent method is presented  in section \S\ref{S:cubature}, and we shall see demonstrated by examples the failure of numerical steepest descent when applied to the integrals of our interest, as well as an outline of how to modify the method such that it works. In section \S\ref{S:loc} we shall put these observations on a firmer foundation and derive some results that will aid in constructing efficient methods for these integrals based in the method of steepest descent. Finally, in \S\ref{S:exp}, we demonstrate the resulting methods on one artificial example, and two examples taken directly from applications.

\section{Numerical steepest descent and cubature}\label{S:cubature}

Applying the method of steepest descent to the univariate integral \eqref{eq:HOI}, assuming that $f(x)$ and $g(x)$ can be analytically extended to the complex plane, and that $g(x)$ is monotone, yields \cite{Huybrechs:analyticcontinuation}
\[
\int_a^bf(x)\e{\ii\omega g(x)}{\rm d}x=G(a)-G(b),
\]
where
\begin{equation}\label{eq:sdint}
G(x)=\int_0^\infty f(h_x(p))\e{-\omega p}h_x'(p){\rm d}p,
\end{equation}
where in turn $h_x(p)$ is the path of steepest descent, implicitly defined through the equation.
\[
g(h_x(p))=g(x)+\ii p.
\]
Note that the requirement that $f(x)$ and $g(x)$ are analytic can be relaxed to analytic in appropriate regions of the complex plane \cite{Huybrechs:2012gm}. More generally a singularity in the complex plane might introduce an exponentially subdominant term \cite{Huybrechs:analyticcontinuation}, and
\[
\int_a^bf(x)\e{\ii\omega g(x)}{\rm d}x=G(a)-G(b)+\mathcal{O}(\e{-c\omega}).
\]
In the following the notation $\mathcal{O}(\e{-c\omega})$ will be used to indicate any exponentialy subdominant term that might be present, depending on the analytic properties of the integrand.

For small $p$, the path of steepest descent will behave like
\[
h_x(p)\sim \sqrt[\alpha]{\frac{\ii \alpha!}{g^{(\alpha)}(x)}p},
\]
i.e., some brach of this expression, where $\alpha-1$ is the number of vanishing derivatives of $g(x)$ at the point $x$ \cite{Asheim:2010kc}. An important consequence of this is that a $p^{-(\alpha-1)/\alpha}$-type singularity is introduced in the steepest descent integral \eqref{eq:sdint}. A substitution $p\to q^\alpha$ will produce an analytic integrand, given that $f(x)$ and $g(x)$ are analytic,
\begin{equation}\label{eq:sdint2}
G(x)=\alpha \int_0^\infty f(h_x(q^\alpha))\e{-\omega q^\alpha}h_x'(q^\alpha)q^{\alpha-1}{\rm d}q
\end{equation}
\[
=\frac{\alpha}{\omega}\int_0^\infty f(h_x(\frac{t^\alpha}{\omega}))h_x'(\frac{t^\alpha}{\omega})t^{\alpha-1}\e{-t^\alpha}{\rm d}t.
\]
The last integral is on a form that can numerically be resolved efficiently by Gaussian quadrature. Given a quadrature rule with $m$ points and weights $\{x_j^\alpha,w_j^\alpha\}_{j=1}^m$ which is Gaussian with respect to the weight $\e{-x^\alpha}$ on $[0,\infty)$, i.e., Gauss-Legendre for $\alpha=1$, and half-space Gauss-Hermite for $\alpha=2$. By Lemma 1 from \cite{Deano:2009jc} the error of this approximation applied to the integral \eqref{eq:sdint2} is $\mathcal{O}(\omega^{-(2m-1)/\alpha})$.

\subsection{Cubature}

The method of steepest descent is inherently a procedure for univariate integrals; no higher dimensional extension of the paths of steepest descent are known. The method can however be developed into a cubature rule by regarding a double integral as nested univariate integration. Huybrechs \& Vandewalle in \cite{Huybrechs:multivariate} demonstrate this approach on cases with stationary points, resonance points and corner points, giving a method with high asymptotic order.

In the following example we shall however see that the numerical method fails on the particular type of integral that we are interested in.
\begin{example}\label{ex:1}
From \cite{Huybrechs:multivariate} we have
\[
\int_0^a\int_0^b f(x,y)\e{\ii\omega\sqrt{x^2+y^2}}{\rm d}y{\rm d}x=F(0,0)-F(a,0)-F(0,b)+F(a,b)+\mathcal{O}(\e{-c\omega}).
\]
where,
\begin{equation}\label{eq:F}
F(x,y)=\!\int_0^\infty\!\!\! \int_0^\infty \!\!\! f(u_{x,y}(p),v_y(u_{x,y}(p),q)) u_{x,y}'(p)\frac{\partial v_y}{\partial q}(u_x(p),q)\e{-\omega(p+q)}{\rm d}p{\rm d}q.
\end{equation}
The paths of steepest descent for the inner integral are easily obtained, and we get
\[
v_y(x,q)=\sqrt{y^2-q^2+2\ii q \sqrt{x^2+y^2}},
\]
\[
u_{x,y}(p)=\sqrt{x^2-p^2+2\ii p \sqrt{x^2+y^2}}.
\]
Now, assume $f(x,y)=1$ and let us then study the contribution from the origin in particular,
\[
F(0,0)=-\int_0^\infty \int_0^\infty \frac{p+q}{\sqrt{2pq+q^2}}\e{-\omega(p+q)}{\rm d}p{\rm d}q.
\]
It can be showed that $F(0,0)=-\frac{\pi}{2\omega^2}$. There is a sqare-root singularity at $q=0$, which follows from the fact that the oscillator function $\sqrt{x^2+y^2}$ has a line of stationary points along $y=0$. Substituting $q\to q^2$, will resolve this,
\[
F(0,0)=-2\int_0^\infty \int_0^\infty \frac{p+q^2}{\sqrt{2p+q^2}}\e{-\omega(p+q^2)}{\rm d}p{\rm d}q
\]
\[
=-2\omega^{-2}\int_0^\infty \int_0^\infty \frac{t_1+t_2^2}{\sqrt{2t_1+t_2^2}}\e{-(t_1+t_2^2)}{\rm d}t_1{\rm d}t_2.
\]
Now it's apparent that this method can never have asymptotic order higher than $2$ unless this integral is resolved exactly by the quadrature method. In other words, the relative error will generally not decrease with $\omega$. Moreover, a pole is present at $q=\pm \ii \sqrt{2p}$, which will also reduce the efficiency of any standard quadrature quadrature rule, like e.g. Gauss-quadrature.
\end{example}

The failure of the steepest descent method in this case should not come as a surprise, since the oscillator function $\sqrt{x^2+y^2}$ is not differentiable at the origin. Indeed, the line of stationary points on the $y$-axis actually degenerates into a regular endpoint as $x\to0$. Now, the observant reader has probably already realised that a polar change of variables would swiftly resolve the problem in this example. One way of seeing this is by observing that, in the sense of a regularised integral,
\[
F(0,0)=\int_0^\infty\int_0^\infty \e{\ii\omega\sqrt{x^2+y^2}}{\rm d}y{\rm d}x=\int_0^{\pi/2}\int_0^\infty \e{\ii\omega r}r{\rm d}r{\rm d}\theta,
\]
which is, of course, the way one arrives at the exact value of $-\frac{\pi}{2\omega^2}$. However, this form also turns out to be much better suited for computations. Let us remove the assumption that $f(x,y)=1$. Now, again in sense of a regularised integral,
\[
F(0,0)=\int_0^\infty\int_0^\infty f(x,y)e^{i\omega\sqrt{x^2+y^2}}{\rm d}x{\rm d}y=\int_0^{\pi/2}\int_0^\infty \tilde{f}(r,\theta)e^{i\omega r}r{\rm d}r{\rm d}\theta,
\]
where
\[
\tilde{f}(r,\theta):=f(r\cos\theta,r\sin\theta).
\]
The method of steepest descent applied to the inner integral gives, assuming that $\tilde{f}(r,\theta)$ is  analytic in a sufficiently large part of the complex plane in its first argument,
\[
F(0,0)=i\int_0^{\pi/2}\int_0^\infty e^{-\omega p}\tilde{f}(ip,\theta){\rm d}p{\rm d}\theta.
\]
Note that the outer integral is non-oscillatory as a function of $\theta$. Thus, the outer integral can be resolved with a classical quadrature rule. The inner integral is efficiently handled by scaled Gauss-Laguerre quadrature. In other words, the polar change of variables has effectively confined the oscillatory behaviour to the radial dimension. In the following we shall see that this is not limited to this particular oscillator function. Also note that if $f(x,y)\sim (x^2+y^2)^{-\nu/2}$ near the origin, then $\tilde{f}(r,\theta)\sim r^{-\nu+1}$, thus cancelling the $1/r$-type singularity typically seen in scattering integrals. Finally note that the integration range of the $\theta$-integration reflects the shape of the domain around the origin. If the origin was contained in the domain, then the outer integration range would be $[0,2\pi]$, and we would be integrating a periodic function, which means that a simple trapezoidal rule would be the natural choice of quadrature in this case.

\section{Theoretical results}\label{S:loc}

The method outlined in the example in the previous section works under the premise that the contribution from the special point $x_0$ can be separated, and that it is of the form 
\[
\int_0^\beta\int_0^\infty \tilde{f}(r,\theta)e^{i\omega\tilde{g}(r,\theta)}{\rm d}r{\rm d}\theta.
\]
In this section we shall prove some fairly general results regarding the application of the method of steepest descent in conjunction with a polar change of variables. The results are formulated as a kind of \emph{pre-quadrature rule}, by which we mean that only the problematic part has been treated. No choice of quadrature rule has been made for the rest, but existing methods are here applicable. Though most real-world cases are with $n=2,3$, the results are formulated for integration in $\mathbb{R}^n$, since this implies little extra work. The tool here is \emph{$n$-spherical coordinates} \cite{Blumenson:1960cj}, which is the higher dimensional equivalent of the well-known polar and spherical coordinates. Here the substitution takes the form,
\begin{align*}
&x_j\to r\cos(\varphi_j)\prod_{l=1}^{j-1}\sin(\varphi_l),\qquad j=1,\hdots,n-1,\\
&x_{n}\to r \prod_{l=1}^{n-1}\sin(\varphi_l),
\end{align*}
where $\varphi_j\in[0,\pi]$, $j=1,\hdots, n-2$, and $\varphi_{n-1}\in[0,2\pi]$. In vectorial form this can be written 
\[
\vc{x}\to r\Theta,\qquad \Theta\in\mathbb{S}^{n-1},
\]
where $\mathbb{S}^{n-1}$ denotes the $(n-1)$-sphere. The volume element has the simple form $r^{n-1}{\rm d}\Theta$, where,
\[
{\rm d}\Theta = \prod_{l=1}^{n-1}\sin^{n-1-l}(\varphi_l){\rm d}\varphi_l.
\]

In the following, we shall assume that the special point $\vc{x}_0=0$, and that $g(\vc{x_0})=0$. Note that this implies no lack of generality, since this represents a simple translation of the coordinate system, plus a scaling of the integral: $\int f(\vc{x})\exp(\ii\omega g(\vc{x})){\rm d}x=\exp(\ii\omega g(\vc{x_0}))\int f(\vc{x_0+x})\exp(\ii\omega (g(\vc{x_0+x})-g(\vc{x_0})){\rm d}x$.

\subsection{Unbounded integrals}

For the unbounded case we consider a domain of integration that is an infinite cone defined in terms of a subset of the $(n-1)$-sphere, $W\subset \mathbb{S}^{n-1}$. Note that this case includes the case of the entire Euclidean space as a special case.

\begin{lemma}\label{thm:unbounded}
For $n>1$, assume that the integral,
\[
I[f]:=\int_{D} f(\vc{x})e^{\ii \omega g(\vc{x})}{\rm d}\vc{x},
\]
exists, with $D$ being an infinite cone: $D=\{r\Theta,\ {\rm where}\ \Theta\in W\subset\mathbb{S}^{n-1}, \ r\in[0,\infty] \}$. Suppose that for the oscillator function $g(\vc{x})$ the following conditions hold whenever $\vc{x}\in D/\{0\}$:
\begin{enumerate}
\item
$g(\vc{x})$ is continuously differentiable.
\item 
$\nabla g(\vc{x})\neq 0.$
\item
$ \frac{\partial}{\partial r}g(r\Theta)>0$.
\end{enumerate}
Assume in addition that $r^{n-1}f(r\Theta)$ and $g(r\Theta)$, $\Theta\in W$, are analytic functions of $r$ in a neighbourhood of the origin.
Then 
\begin{equation}\label{eq:radial}
I[f] = \frac {1} n\int_{W}\int_0^\infty f(\rho_0(p,\Theta)\Theta)\frac{\partial (\rho_0^n)}{\partial p}(p,\Theta)\e{-\omega p}{\rm d}p{\rm d}\Theta+\mathcal{O}(\e{-c\omega}),
\end{equation}
where $\rho_0(p,\theta)$ satisfies the equation
\[
g(\rho_0(p,\Theta)\Theta)=\ii p.
\] 
\end{lemma}
\begin{proof}
Writing the integral in $n$-spherical coordinates, we get,
\[
I[f] = \int_{W}\int_0^\infty f(r\Theta)\e{\ii\omega g(r\Theta)}r^{n-1}{\rm d}r{\rm d}\Theta.
\]
Let us apply the method of steepest descent to the inner integral. Now conditions 1, 2, and 3 on $g(\vc{x})$ ensures that there are no contributing points of $g(r\Theta)$ for $r>0$, and condition 3 ensures that $\rho_0(p,\Theta)$ exists for all $\Theta$, thus,
\begin{multline*}
\int_0^R f(r\Theta)\e{\ii\omega g(r\Theta)}r^{n-1}{\rm d}r= 
\frac{1}{n}\int_0^\infty f(\rho_0(p,\Theta)\Theta)\frac{\partial (\rho_0^n)}{\partial p}(p,\Theta)\e{-\omega p}{\rm d}p\\
-\frac{\e{i\omega g(R\Theta)}}{n}\int_0^\infty f(\rho_R(p,\Theta)\Theta)\frac{\partial (\rho_R^n)}{\partial p}(p,\Theta)\e{-\omega p}{\rm d}p + \mathcal{O}(\e{-c\omega})
\end{multline*}
where $g(\rho_R(p,\Theta)\Theta)=g(R\Theta)+\ii p$. Letting $R\to\infty$ the second integral must vanish. The result follows. 
\end{proof}

From this Lemma a very simple, and potentially highly efficient, quadrature rule, or more precisely a pre-quadrature rule, for these particular kinds of integrals follows.

\begin{theorem}\label{thm:unbounded_rule}
Assume the conditions of Lemma \ref{thm:unbounded} are satisfied. Let $\alpha$ be the smallest integer such that
\begin{equation}\label{eq:cond1}
\frac{\partial^l g(r\Theta)}{\partial r^l}\Big{|}_{r=0+}= 0.\quad \forall \Theta\in W\subset\mathbb{S}^{n-1},\quad l=0,1,\hdots,\alpha-1,
\end{equation}
and assume that 
\begin{equation}\label{eq:cond2}
\frac{\partial^\alpha g(r\Theta)}{\partial r^\alpha}\Big{|}_{r=0+}>0,\quad \forall \Theta\in W\subset\mathbb{S}^{n-1}.
\end{equation}
Let $\{x_j^\alpha,w_j^\alpha\}_{j=1}^m$ be the weights and nodes of the $m$-point rule being Gaussian with respect to the weight $\e{-x^\alpha}$ on $[0,\infty)$. Define,
\begin{equation}\label{eq:unbounded_rule}
Q_r[f](\Theta) = \frac {\alpha} {n\omega}\sum_{j=1}^m w_jx_j^{\alpha-1}f(\rho_0(x_j^\alpha/\omega,\Theta)\Theta)\frac{\partial (\rho_0^n)}{\partial p}(x_j^\alpha/\omega,\Theta).
\end{equation}
Then
\[
I[f]-\int_{W}Q_r[f](\Theta){\rm d}\Theta=\mathcal{O}(\omega^{-\frac{2m-1}\alpha}).
\]

\end{theorem}

\begin{proof}
Departing from Eq. \eqref{eq:radial}, we consider the inner integral
\[
I_{inner}(\Theta):=\frac 1n \int_0^\infty f(\rho_0(p,\Theta)\Theta)\frac{\partial (\rho_0^n)}{\partial p}(p,\Theta)\e{-\omega p}{\rm d}p.
\]
From of conditions \eqref{eq:cond1} and \eqref{eq:cond2} it follows that $\rho(p,\Theta)$ has an expansion of the form
\[
\rho_0(p,\Theta)\sim a_1(\Theta)p^{1/\alpha}+a_2(\Theta)p^{2/\alpha}+\hdots,
\] 
valid for all $\Theta\in W$. Introducing the change of variables $p\to q^\alpha$ we get 
\[
I_{inner}(\Theta)=\frac \alpha n\int_0^\infty f(\rho_0(q^\alpha,\Theta)\Theta)\frac{\partial \rho_0^n}{\partial p}(q^\alpha,\Theta)q^{\alpha-1}\e{-\omega q^\alpha} {\rm d}q,
\]
whose integrand is analytic around $q=0$. Now, applying Lemma 1 of \cite{Deano:2009jc},
\[
I_{inner}(\Theta)-\frac {\alpha} {n\omega}\sum_{j=1}^m w_jx_j^{\alpha-1}f(\rho_0(x_j^\alpha/\omega,\Theta)\Theta)\frac{\partial (\rho_0^n)}{\partial p}(x_j^\alpha/\omega,\Theta)= \mathcal{O}(\omega^{-\frac{2m-1}{\alpha}}).
\]
Integrating over $W$ yields the sought conclusion.
\end{proof}

\begin{remark}
Note here that $\frac{\partial (\rho_0^n)}{\partial p}\sim q^{n-1}$, so in the case where $f(\vc{x})$ is regular at the origin, applying a Gaussian rule with the weight $x^{n-1}\e{-x^\alpha}$ will yield higher asymptotic accuracy.
\end{remark}

The following corollary follows directly from the above theorem, and implies that when a quadrature method is applied for the outer integration, the error will decrease asymptotically with $\omega$ down to the error level of the (non-oscillatory) outer integration. 

\begin{corollary}\label{corr:1}
Let $Q_W:W\to\mathbb{R}$ be a quadrature rule,
\[
\int_W h(\Theta){\rm d}\Theta=Q_W[h]+E[h].
\]
Then
\[ 
I[f]-Q_W[Q_r[f]]=E[Q_r[f]]+\mathcal{O}(\omega^{-\frac{2m-1}\alpha}).
\]
\end{corollary}

\subsection{Bounded integrals}

Next we shall investigate the case of integration over a more general bounded domain $D\subset\mathbb{R}^n$,
\[
I[f]:=\iint_{D} f(\vc{x})e^{\ii \omega g(\vc{x})}{\rm d}\vc{x}.
\]
Assume that $D$ is star shaped with respect to the origin,  i.e., the segment $(0,x)$ is contained in $D$ for all $x\in\partial D$, and we can write $D=\{r\Theta, \ \forall r\in[0,R(\Theta)] ,\  \Theta\in W\subset\mathbb{S}^{n-1}\}$.

In $n$-spherical coordinates centred at the origin we then have,
\[
I[f]=\iint_{W}\int_0^{R(\Theta)} f(r\Theta)\e{\ii\omega g(r\Theta)}r^{n-1}{\rm d}r{\rm d}\Theta.
\]
Assuming the functions $r^{n-1}f(r\Theta)$ and $g(r\Theta)$ are analytic as a function of $r$ in a neighbourhood around the origin, just as in the proof of Lemma \ref{thm:unbounded}, we can apply the method of steepest descent and get
\begin{multline}\label{eq:bounded_sd}
I[f]=\frac {1} n\iint_{W}\int_0^\infty f(\rho_0(p,\Theta)\Theta)\frac{\partial (\rho_0^n)}{\partial p}(p,\Theta)\e{-\omega p}{\rm d}p{\rm d}\Theta \\
-\frac {1} n\iint_{W} \e{\ii \omega R(\Theta)}\int_0^\infty f(\rho_R(p,\Theta)\Theta)\frac{\partial (\rho_R^n)}{\partial p}(p,\Theta)\e{-\omega p}{\rm d}p{\rm d}\Theta + \mathcal{O}(\e{-c\omega}),
\end{multline}
where $\rho_R(p,\Theta)$ is defined through the equation $g(\rho_R(p,\Theta))=R(\Theta)+\ii p$. This leads to a theorem of the same kind as Theorem \ref{thm:unbounded_rule}, but with an additional term that includes the contributions from the boundary. 

\begin{theorem}\label{thm:bounded}
Assume $D$ is of the form $D=\{t\Theta, \ \forall t\in[0,R(\Theta)] ,\  \Theta\in W\subset\mathbb{S}^{n-1}\}$, $\forall\Theta\in W\subset\mathbb{S}^{n-1}$. Furthermore assume that the conditions 1, 2, and 3 in Lemma \ref{thm:unbounded} hold for $\vc{x}\in D /\{0\}$, and that conditions \eqref{eq:cond1} and \eqref{eq:cond2} hold. Let $\{x_j^\alpha,w_j^\alpha\}_{j=1}^m$ be the weights and nodes of the $m$-point rule being Gaussian with respect to the weight $\e{-x^\alpha}$ on $[0,\infty)$ and let $Q_r[f]$ be given by equation \eqref{eq:unbounded_rule}. Now define
\begin{equation}\label{eq:bounded_rule}
\tilde{Q}_r[f](\Theta) = Q_r[f](\Theta)-\frac {1} n \e{\ii \omega R(\Theta)}\int_0^\infty f(\rho_R(p,\Theta)\Theta)\frac{\partial (\rho_R^n)}{\partial p}(p,\Theta)\e{-\omega p}{\rm d}p.
\end{equation}
then 
\[
\iint_{W}\tilde{Q}_r[f](\Theta){\rm d}\Theta-I[f]=\mathcal{O}(\omega^{-(2m-1)/\alpha})
\]
\end{theorem}
\begin{proof}
Under the given assumptions, we have the that Eq. \eqref{eq:bounded_sd} holds, and we simply apply Theorem \ref{thm:unbounded_rule}, and the result follows.
\end{proof}

Let us end this section with some remarks

\begin{remark}
Only for the central contribution will we see the oscillations being confined to the radial direction, and in this case must the remaining integral be treated with oscillatory quadrature techniques with sufficiently high order in order to retain the asymptotic error decay indicated in here. In practice, for the method of steepest descent to be directly applicable to the remainder we need that $R(\Theta)>0$, $\forall \Theta\in W$. That means cases where the origin is contained in the interior of $D$, and intersections of such sets with cones, as are discussed in Theorem \ref{thm:unbounded}. Note that this stronger condition does not allow the origin to be on a curved section of the boundary.
\end{remark}

\begin{remark}
Theorem \ref{thm:bounded} gives a notably different decomposition that what was seen in Example \ref{ex:1}. In the example the method of steepest descent is applied to the integral given in it's original coordinates, and then the central contribution is replaced by an unbounded integral on polar form. However, if the method of steepest descent is applied to the remainder in \eqref{eq:bounded_rule}, this will result in a decomposition like that of Example \ref{ex:1}, with contributions being local to each corner of the rectangle. The contributions will be given in different coordinates, but at least asymptotically the decompositions must be the same. This will be touched upon in the following section.
\end{remark}

\section{Numerical experiments}\label{S:exp}
\begin{figure}[htbp]
\begin{center}
\includegraphics[width=.8\textwidth]{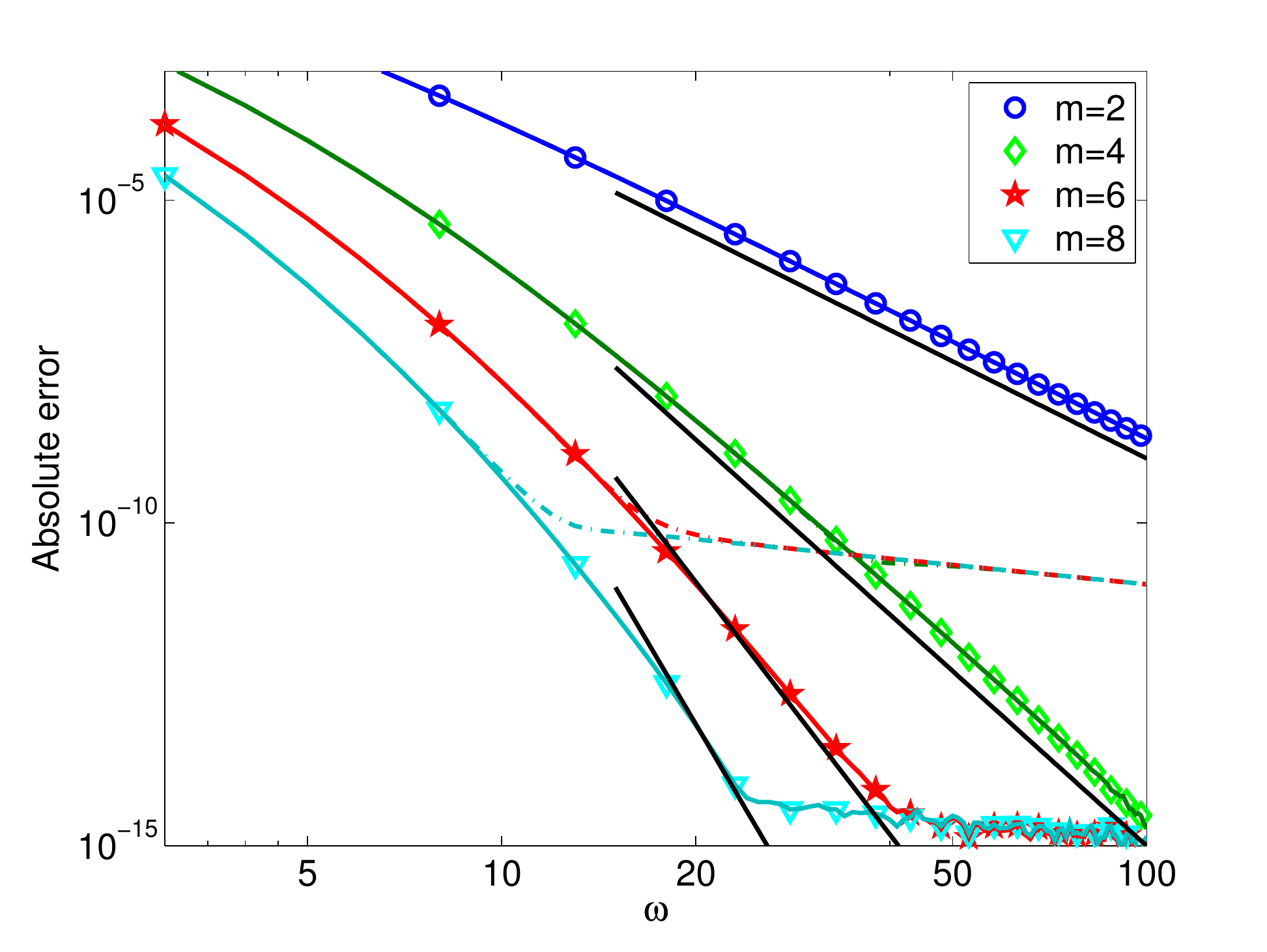}
\caption{Absolute errors in resolving the integral \eqref{eq:exp1}. Straight lines indicate lines of $\sim \omega^{-5}$, $\sim \omega^{-9}$, $\sim \omega^{-13}$ and $\sim \omega^{-17}$, respectively. Dashed lines indicate result of same experiment with $N=30$ points in each dimension for the outer integration.}
\label{fig:exp1}
\end{center}
\end{figure}

Let us first consider an artificial example. The following integral has a relatively simple closed form solution,
\begin{equation}\label{eq:exp1}
\int_{-\infty}^\infty \int_{-\infty}^\infty \int_{-\infty}^\infty  \frac{\e{\ii\omega \sqrt{x^2+2y^2+3z^2}}}{(x^2+2y^2+3z^2)(1+\sqrt{x^2+2y^2+3z^2})}{\rm d}z{\rm d}y{\rm d}x,
\end{equation}
\[
=\sqrt{\frac{2}{3}} \pi  (i \cos(\omega)+\sin(\omega)) (\pi +2 i \text{Ci}(\omega)-2 \text{Si}(\omega)).
\]
The spherical substitution is $x\to r\cos \varphi_1$, $y\to r\sin\varphi_1\cos \varphi_2$, $z\to r\sin\varphi_1\sin\varphi_2$, never mind that a scaling of each dimension would in this case produce a simpler expression. Now, in spherical variables,
\[
g(r,\varphi_1,\varphi_2)=r \sqrt{\cos^2(\varphi_1) + \frac12 (5 + \cos(2 \varphi_2))\sin^2(\varphi_1)},
\]
and then
\[
\rho(p,\varphi_1,\varphi_2)=\frac{ip}{\sqrt{\cos^2(\varphi_1) + \frac12 (5 + \cos(2 \varphi_2))\sin^2(\varphi_1)}}.
\]
Theorem \ref{thm:unbounded} indicates that a method with high asymptotic order is obtained by applying a Gaussian rule in the radial direction. For the given integral we apply formula \eqref{eq:unbounded_rule}, where we use Clenshaw-Curtis quadrature with $50$ points in the $\varphi_1$-direction, and the trapezoidal rule with $50$ points in the (periodic) $\varphi_2$-direction for integrating over the sphere. The function that here is being integrated is visualised in Fig. \ref{fig:exp12}. The result of this experiment can be seen in Fig. \ref{fig:exp1}, where the method has been applied with the number of complex Gaussian points in the radial integration being between $2$ and $8$. One observes an asymptotic decrease in the error matches the prediction of Theorem \ref{thm:unbounded}. Observe that machine precision can be attained with relatively few quadrature points when $\omega$ is sufficiently large. In Fig. \ref{fig:exp1} curves are also included that show the result of using less points ($30$) for the outer integration. These curves match what should be expected from Corollary \ref{corr:1}.

\begin{figure}[htbp]
\begin{center}
\includegraphics[width=.8\textwidth]{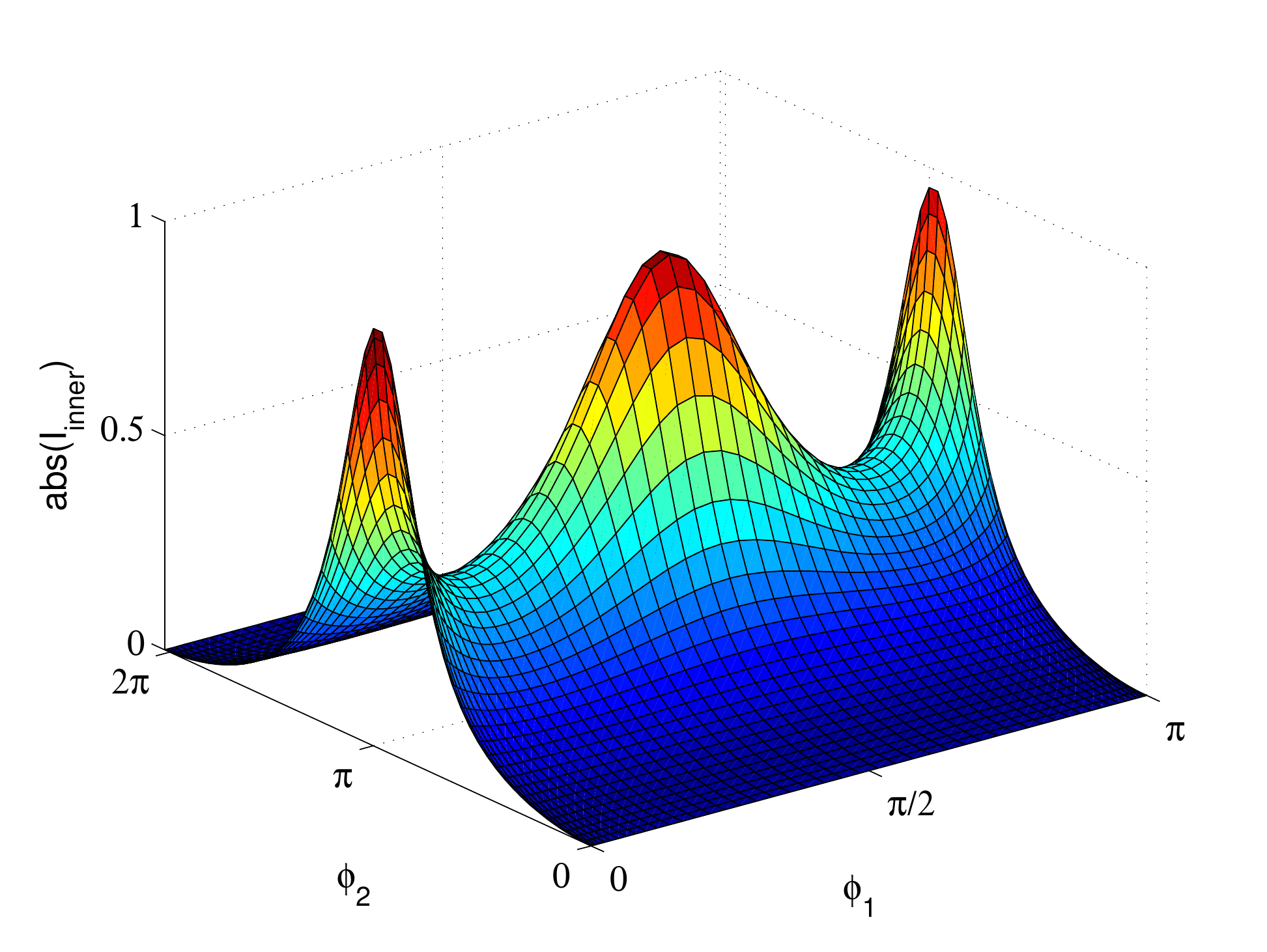}
\caption{The absolute value of $I_{inner}(\theta,\varphi)$, for $k=100$.}
\label{fig:exp12}
\end{center}
\end{figure}

\subsection{A problem from acoustics}

\begin{figure}
\begin{center}
\includegraphics[width=8cm]{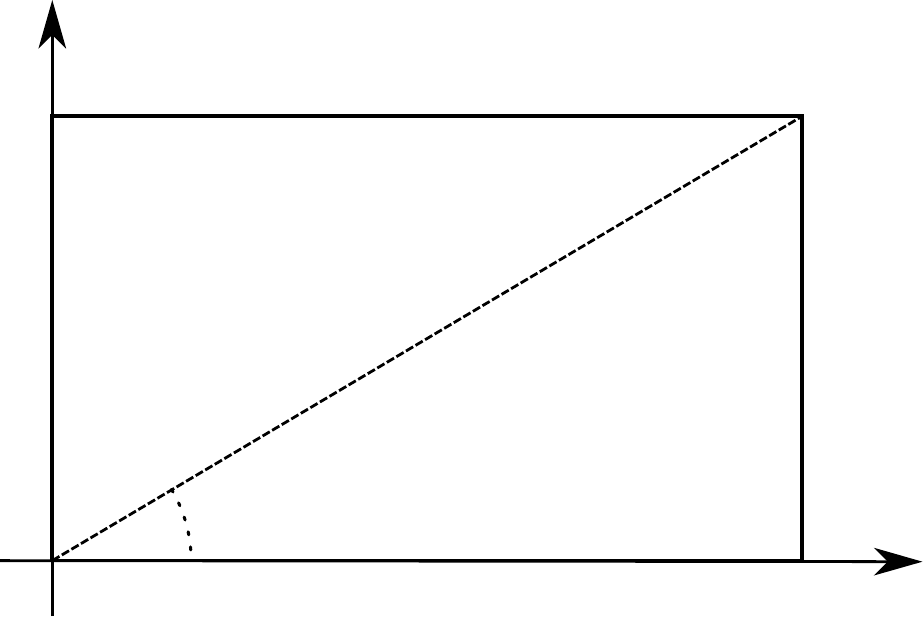}
\caption{Splitting of the outer integration.}
\label{fig:splitt}
\begin{picture}(0,0)(0,0)
\put(-63,57){$\beta$}
\put(80,40){$a$}
\put(-110,155){$b$}
\put(0,100){$\eta$}
\end{picture}
\end{center}
\end{figure}
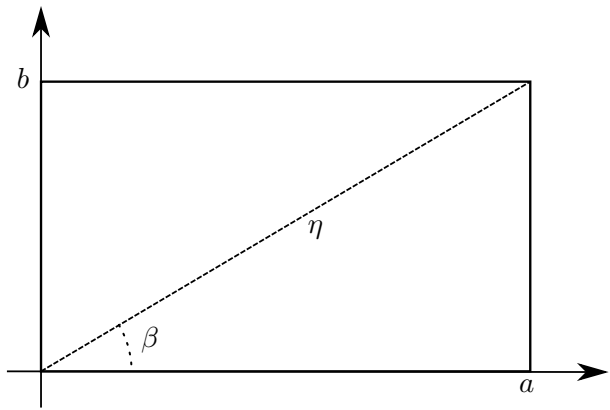

For the problem of a rectangular duct terminating in an infinite baffle, sound pressure at the end of the duct expressed in terms of the axial velocity of the baffle, $v(x,y)$ is of the form (see \cite{Kemp:2001vj} for details),
\[
p(x_0,y_0)=C\int_0^a\int_0^b\frac{e^{\ii\omega\sqrt{(x-x_0)^2+(y-y_0)^2}}}{\sqrt{(x-x_0)^2+(y-y_0)^2}}v(x,y){\rm d}y{\rm d}x,
\]
for some constant $C$. We shall consider a model case of this integral,
\[
I[f]=\int_0^a\int_0^b\frac{e^{\ii\omega\sqrt{x^2+y^2}}}{\sqrt{x^2+y^2}}f(x,y){\rm d}y{\rm d}x,
\]
For simplicity we denote $\tilde{f}(r,\theta)=f(r\cos(\theta),r\sin(\theta))$. The central contribution is now computed from \eqref{eq:unbounded_rule}
\[
Q_r[f]=\frac{\ii}{\omega}\sum_{j=1}^m w_j\tilde{f}(\tfrac{\ii x_j}{\omega},\theta),
\]
and an approximation to the full integral is, by By Theorem \ref{thm:bounded},
\[
\int_0^{\pi/2}\tilde{Q}_r[f]{\rm d}\theta=\int_0^{\pi/2}Q_r[f]{\rm d}\theta-I_{ext}[f],
\]
where
\[
I_{ext}[f]=\ii\int_0^{\pi/2}\e{\ii\omega R(\theta)}\int_0^\infty \tilde{f}\big{(}R(\theta)+\ii p,\theta\big{)}\e{-\omega p}{\rm d}p{\rm d}\theta.
\]
In order to compute the last integral, we apply the method of steepest descent again, noting that the integral should be split in the $\theta$ direction. Let in the following $\beta=\arctan\frac b a$, as illustrated in Fig \ref{fig:splitt}. 
then $I_{ext}[f]=I_{ext}^1[f]+I_{ext}^2[f]$ with
\[
I_{ext}^1[f]=\ii\int_0^{\beta}\e{\ii\omega a\sec(\theta)}\int_0^\infty \tilde{f}\big{(}a\sec(\theta)+\ii p,\theta\big{)}\e{-\omega p}{\rm d}p{\rm d}\theta,
\]
and
\[
I_{ext}^2[f]=\ii\int_\beta^{\pi/2}\e{\ii\omega b\csc(\theta)}\int_0^\infty \tilde{f}\big{(}b\csc(\theta)+\ii p,\theta\big{)}\e{-\omega p}{\rm d}p{\rm d}\theta.
\]
We compute two paths of steepest descent for the $\theta$-integration in each of the integrals,
\[
h_{1,1}(q)=\sec^{-1}(1+\ii q/a),\qquad h_{1,2}(q)=\sec^{-1}\left(\frac{\eta+\ii q}a\right),
\]
\[
h_{2,1}(q)=\csc^{-1}\left(\frac{\eta+ \ii q}b\right),\qquad h_{1,2}(q)=\csc^{-1}(1+\ii q/b).
\]
where $\eta:=a\sec\beta=b \csc\beta=\sqrt{a^2+b^2}$. Thus
\[
I_{outer}[f]=I_{outer}^1[f]+I_{outer}^2[f]=I_{outer}^{1,1}[f]-I_{outer}^{1,2}[f]+I_{outer}^{2,1}[f]-I_{outer}^{2,2}[f],
\]
where
\[
I_{ext}^{1,1}[f]=-a\e{\ii\omega a}\int_0^{\infty}\int_0^\infty \frac{\tilde{f}\big{(}a+\ii q+\ii p,\sec^{-1}(1+\ii q/a)\big{)}}{(a + \ii q)\sqrt{2\ii q a-q^2} }\e{-\omega (p+q)}{\rm d}p{\rm d}q,
\]
\[
I_{ext}^{1,2}[f]=-a\e{\ii\omega \eta}\int_0^{\infty}\int_0^\infty \frac{\tilde{f}\big{(}\eta+\ii q+\ii p\sec^{-1}\left(\frac{\eta+\ii q}a\right)\big{)}}{(\eta + \ii q)\sqrt{b^2-q^2+2\ii q \eta} }\e{-\omega (p+q)}{\rm d}p{\rm d}q,
\]
\[
I_{ext}^{2,1}[f]=b\e{\ii\omega \eta}\int_0^{\infty}\int_0^\infty \frac{\tilde{f}\big{(}\eta+\ii q+\ii p,\csc^{-1}\left(\frac{\eta+\ii q}b\right)\big{)}}{(\eta + \ii q)\sqrt{a^2-q^2+2\ii q \eta} }\e{-\omega (p+q)}{\rm d}p{\rm d}q,
\]
\[
I_{ext}^{2,2}[f]=b\e{\ii\omega b}\int_0^{\infty}\int_0^\infty \frac{\tilde{f}\big{(}b+\ii q+\ii p,\csc^{-1}(1+\ii q/b)\big{)}}{(b + \ii q)\sqrt{2\ii q b-q^2} }\e{-\omega (p+q)}{\rm d}p{\rm d}q.
\]
Now one should note that the integrals $I_{outer}^{1,1}[f]$ and $I_{outer}^{2,2}[f]$ have a square-root singularity at $q=0$, this is because the points $[a,0]$ and $[0,b]$ are resonance points, which induce stationary points in $R(\theta)$. Substituting $p\to p/\omega$ and $q\to q^2/\omega$ in these integrals, and $p\to p/\omega$ and $q\to q/\omega$ in the other two integrals bring them on a form that can efficiently be evaluated with tensor-product half-space Gauss-Hermite $\times$ Gauss Laguerre, and Gauss-Laguerre $\times$ Gauss Laguerre.

To test the method we apply it to the following problem,
\[
\int_0^1\int_0^2\frac{e^{\ii \omega\sqrt{x^2+y^2}}}{\sqrt{x^2+y^2}}y\cos(x){\rm d}y{\rm d}x
=\frac{\ii}{\omega}\int_0^1 (e^{\ii\omega x}-e^{\ii\omega \sqrt{1+x^2}}) \cos(x){\rm d}x.
\]
By integrating the right hand integral to machine precision with \emph{Matlab's} \emph{quadgk} routine we have a reliable reference solution.

In Figure \ref{fig:exp2} we observe that on this problem machine precision can be reached with relatively few points in the complex plane. For producing this figure the non-oscillatory integration was performed with $30$-point Clenshaw-Curtis quadrature. Note that using the decomposition from Example \ref{ex:1} with modified central contribution (dashed lines) gives a very similar picture. 

\begin{figure}[htbp]
\begin{center}
\includegraphics[width=.8\textwidth]{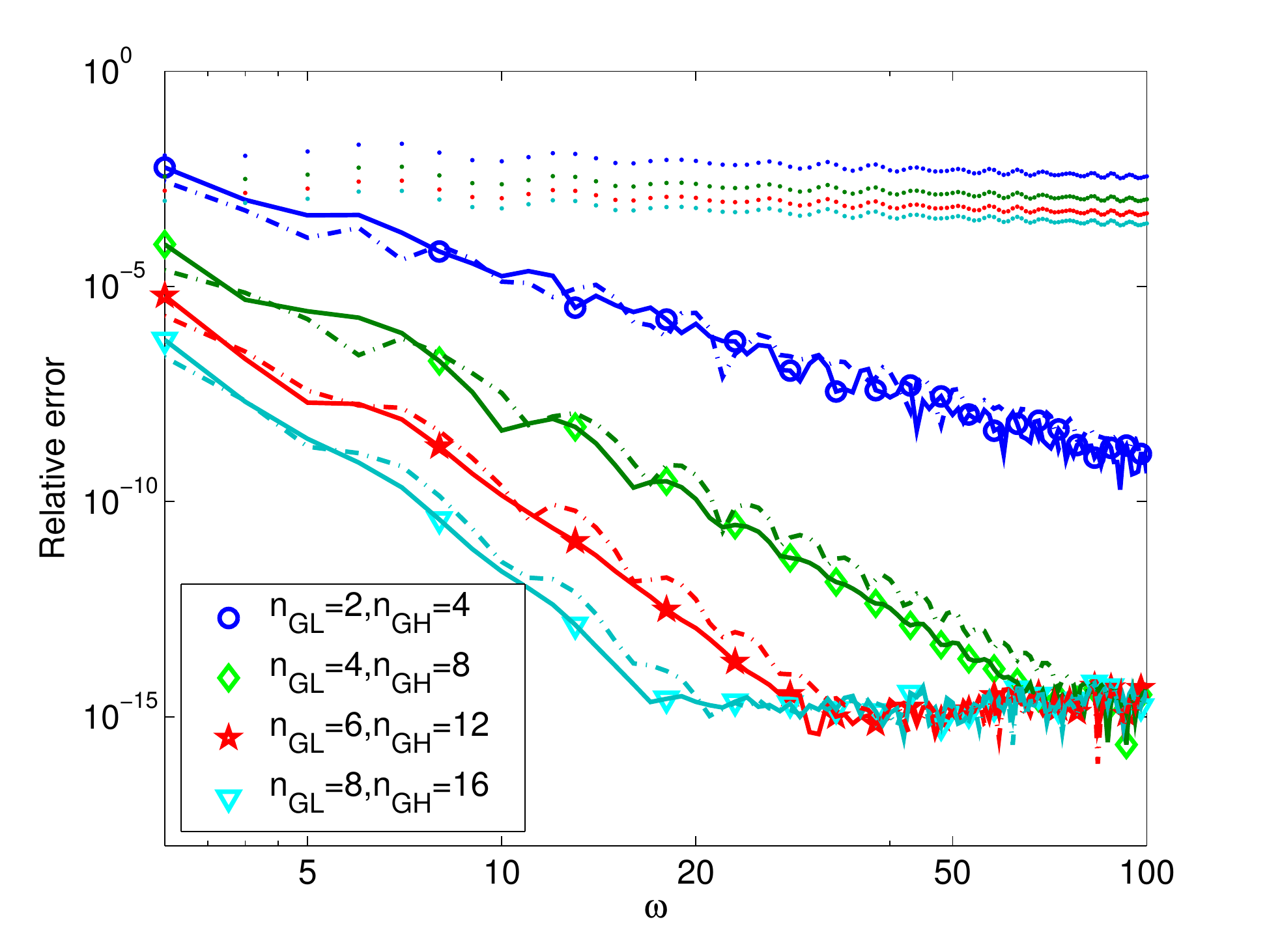}
\caption{Relative error using $n_{GL}=2,4,6,8$ Gauss-Laguerre points and $n_{GH}=2n_{GL}$ half-space Gauss-Hermite points. Dotted lines indicate a direct application of numerical steepest descent (cf. Ex. \ref{ex:1}), and dashed lines idem, but with modified central contribution.}
\label{fig:exp2}
\end{center}
\end{figure}

\subsection{A problem from high frequency scattering}

The scalar exterior scattering problem is often formulated in terms of the Kirchhoff-Helmholtz integral equation for an unknown surface field (see e.g. \cite{nedelec} for details). Here we shall limit the discussion to the case of Dirichlet boundary conditions, where one can use a Fredholm equation of the first kind,
\[
\int_{\partial \Omega} q(\vc{x})G(\vc{x}-\vc{y}){\rm d}\vc{s_x}=-u^i(\vc{y}),\qquad \vc{y}\in\partial\Omega,
\]
where $\partial\Omega$ denotes the surface of the scatterer, $G(\vc{x},\vc{y})$ the free-space Green's function, $u^i(\vc{y})$ is the incident field, and $q(\vc{x})$ is the unknown surface field. We shall in the following discuss the $3D$ case with an incident plane wave, 
\[
G(\vc{x},\vc{y})=\frac{\e{\ii k |\vc{x}-\vc{y}|}}{4\pi |\vc{x}-\vc{y}|},
\]
\[
u^i(\vc{y})=\e{\ii k (\vc{y}\cdot \vc{d})}, \quad {\rm where}\quad ||\vc{d}||=1.
\]
The field $q(\vc{x})$ is in general oscillatory with wavelength $k$, however, as has been demonstrated in e.g. \cite{Dominguez:2007th}, by factoring out the phase of the Kirchhoff approximation, $q(\vc{y})=\tilde{q}(\vc{x})\e{\ii k \vc{x}\cdot\vc{d}}$, the remaining amplitude function $\tilde{q}(\vc{x})$ will be non-oscillatory in illuminated regions of the scatterer. The new unknown satisfies
\[
\int_{\partial \Omega} \tilde{q}(\vc{x})\frac{\e{\ii k (|\vc{x}-\vc{y}|+(\vc{x}-\vc{y})\cdot \vc{d})}}{4\pi |\vc{x}-\vc{y}|}{\rm d}\vc{s_x}=-1,\qquad \vc{y}\in\partial\Omega.
\]
Now, relatively few degrees of freedom are needed to represent $\tilde{q}(\vc{x})$, but a discretisation matrix will necessarily involve integrals with the same kernel as the integral here on the left. 

Studying the case of a spherical scatterer of radius 1 centred at the origin, we set $\vc{x}=[1,0,0]^T$, and investigate the resulting integral for different incident directions $\vc{d}=[d_1,d_2,d_3]$. The corresponding integral equation is
\begin{multline}\label{eq:w0}
\int_0^\pi\int_0^{2\pi} \tilde{q}(\phi_1,\phi_2)
\frac{\e{\ii k(\sqrt{2+2\cos(\phi_2)\sin(\phi_1)}-\sin(\phi_1)\cos(\phi_2)d_1-\sin(\phi_1)\sin(\phi_2)d_2-(1+\cos(\phi_1))d_3)}}{4\pi \sqrt{2+2\cos(\phi_2)\sin(\phi_1)}}
\\ \sin(\phi_1){\rm d}\phi_2{\rm d}\phi_1=-1.
\end{multline}
Using the polar coordinate substitution, centred around $\phi_1=\pi/2$ $\phi_2=0$, on this integral we can show that 
\[
\frac{\partial g(r\Theta)}{\partial r}\big{|}_{r=0^+}=1+[-d_3,d_2]^T\Theta,
\]
and it follows from this that the conditions of Theorem \ref{thm:bounded} are satisfied whenever $d_1\neq 0$, which translates to $\vc{x}$ being on the shadow boundary.
Computing only the central contribution, i.e., $\int_0^{2\pi}Q_r[\phi]{\rm d}\theta$, gives a local approximation of the integral, much like what is achieved with the localised method of stationary phase \cite{Ganesh:2007ee}. A way to test the usefulness of this approximation is to consider a 3D version of the method in \cite{Asheim:2010gk}, whereby a prototype asymptotic Filon-approximation of $\tilde{q}$ would be given as,
\[
\tilde{q}(\vc{x})\sim -\frac 1 w_0.
\]
Here $w_0$ is constructed as a local approximation of the integral of the left hand side of \eqref{eq:w0} with $\tilde{q}(\phi_1,\phi_2)=1$. In Table \ref{tab:exp3} we see the result of such a test, where now $d=[\-\cos(\Psi),0,\sin(\Psi)]^T$, and we compute the relative error w.r.t the analytic solution given by the Mie series \cite{nedelec}. For this table the number of Laguerre points are $m=5$, and the angular direction is computed with a $100$ point trapezoidal rule. This test clearly shows the potential of steepest descent for high feequency scattering computations.

\begin{table}[htdp]
\caption{Relative error of prototype Filon-approximation to $\tilde{q}(\vc{x})$.}
\begin{center}
\begin{tabular}{c|cccc}
$\Psi \backslash$ $k$ & 50 & 100 & 150 & 200\\
\hline
0&5.43380e-04  & 1.37074e-04  & 6.10265e-05  & 3.43482e-05 \\
$\frac \pi {10}$ &2.28316e-03  & 1.11796e-03  & 7.42382e-04  & 5.56018e-04 \\
$\frac \pi {5} $&1.22758e-02  & 6.40899e-03  & 4.31359e-03  & 3.24677e-03\\
$\frac \pi {3} $&4.79852e-02  & 4.05070e-02  & 3.24495e-02  & 2.67878e-02 \\
\end{tabular}
\end{center}
\label{tab:exp3}
\end{table}%

\bibliographystyle{siam}
\bibliography{bibliography}

\begin{thebibliography}{10}

\bibitem{Asheim:2008eb}
{\sc A.~Asheim}, {\em {A combined Filon/asymptotic quadrature method for highly
  oscillatory problems}}, BIT, 48 (2008), pp.~425--448.

\bibitem{Asheim:2010kc}
{\sc A.~Asheim and D.~Huybrechs}, {\em Asymptotic analysis of numerical
  steepest descent with path approximations}, Found. Comput. Math., 10 (2010),
  pp.~647--671.

\bibitem{Asheim:2010gk}
\leavevmode\vrule height 2pt depth -1.6pt width 23pt, {\em {Local solutions to
  high-frequency 2D scattering problems}}, J. Comput. Phys., 229 (2010),
  pp.~5357--5372.

\bibitem{Blumenson:1960cj}
{\sc L.~E. Blumenson}, {\em {Classroom Notes: A Derivation of $n$-Dimensional
  Spherical Coordinates}}, Am. Math. Mon., 67 (1960), pp.~63--66.

\bibitem{Bouwkamp:1954wt}
{\sc C.~Bouwkamp}, {\em {Diffraction theory}}, Rep. Prog. Phys., 17 (1954),
  p.~35.

\bibitem{BRUNO:2007gt}
{\sc O.~Bruno and C.~Geuzaine}, {\em An o(1) integration scheme for
  three-dimensional surface scattering problems}, {J}. {C}omput. {A}ppl.
  {M}ath., 204 (2007), pp.~463--476.

\bibitem{ChandlerWilde:2012kq}
{\sc S.~N. Chandler-Wilde, I.~G. Graham, S.~Langdon, and E.~A. Spence}, {\em
  {Numerical-asymptotic boundary integral methods in high-frequency acoustic
  scattering}}, Acta numerica, 21 (2012), pp.~89--305.

\bibitem{Deano:2009jc}
{\sc A.~Dea{\~n}o and D.~Huybrechs}, {\em Complex {G}aussian quadrature of
  oscillatory integrals}, Numer. Math., 112 (2009), pp.~197--219.

\bibitem{Dominguez:2007th}
{\sc V.~Dominguez and I.~Graham}, {\em {A hybrid numerical-asymptotic boundary
  integral method for high-frequency acoustic scattering}}, Num. Math., 106
  (2007), pp.~471--510.

\bibitem{Dominguez:2011td}
\leavevmode\vrule height 2pt depth -1.6pt width 23pt, {\em {Stability and error
  estimates for Filon--Clenshaw--Curtis rules for highly oscillatory
  integrals}}, IMA J. Numer. Anal.,  (2011).

\bibitem{Ganesh:2007ee}
{\sc M.~Ganesh, S.~Langdon, and I.~H. Sloan}, {\em {Efficient evaluation of
  highly oscillatory acoustic scattering surface integrals}}, J. Comput. Appl.
  Math., 204 (2007), pp.~363--374.

\bibitem{Huybrechs:2012gm}
{\sc D.~Huybrechs and S.~Olver}, {\em {Superinterpolation in highly oscillatory
  quadrature}}, Found. Comput. Math., 12 (2012), pp.~203--228.

\bibitem{Huybrechs:analyticcontinuation}
{\sc D.~Huybrechs and S.~Vandewalle}, {\em On the evaluation of highly
  oscillatory integrals by analytic continuation}, SIAM J. Numer. Anal., 44
  (2006), pp.~1026--1048.

\bibitem{Huybrechs:multivariate}
{\sc D.~Huybrechs and S.~Vandewalle}, {\em The construction of cubature rules
  for multivariate highly oscillatory integrals}, {M}ath. {C}omp., 76 (2007),
  pp.~1955--1980.

\bibitem{Iserles:2004ke}
{\sc A.~Iserles and S.~P. N\o{}rsett}, {\em {On Quadrature Methods for Highly
  Oscillatory Integrals and Their Implementation}}, BIT, 44 (2004),
  pp.~755--772.

\bibitem{Iserles:usingderivatives}
{\sc A.~Iserles and S.~P. N{\o}rsett}, {\em Efficient quadrature of highly
  oscillatory integrals using derivatives}, Proc. Roy. Soc. A., 461 (2005),
  pp.~1383--1399.

\bibitem{Iserles:multivariate_stationary}
{\sc A.~Iserles and S.~P. N{\o}rsett}, {\em {On the computation of highly
  oscillatory multivariate integrals with stationary points}}, BIT, 46 (2006),
  pp.~549--566.

\bibitem{Iserles:multivariate_derivatives}
{\sc A.~Iserles and S.~P. N{\o}rsett}, {\em Quadrature methods for multivariate
  highly oscillatory integrals using derivatives}, Math. Comp., 75 (2006),
  pp.~1233--1258.

\bibitem{Kemp:2001vj}
{\sc J.~A. Kemp, D.~M. Campbell, and N.~Amir}, {\em {Multimodal radiation
  impedance of a rectangular duct terminated in an infinite baffle}}, Acta
  Acustica united with Acustica, 87 (2001), pp.~11--15.

\bibitem{Levin:1996kp}
{\sc D.~Levin}, {\em Fast integration of rapidly oscillatory functions}, J.
  Comput. Appld. Maths., 67 (1996), pp.~95--101.

\bibitem{nedelec}
{\sc J.~N{\'e}d{\'e}lec}, {\em Acoustic and electromagnetic equations: Integral
  representations for harmonic problems}, vol.~144, Springer, NY, 2001.

\bibitem{Olver:2005jp}
{\sc S.~Olver}, {\em {Moment-free numerical integration of highly oscillatory
  functions}}, IMA J. Numer. Anal., 26 (2005), pp.~213--227.

\bibitem{Olver:2006iq}
\leavevmode\vrule height 2pt depth -1.6pt width 23pt, {\em {On the Quadrature
  of Multivariate Highly Oscillatory Integrals Over Non-polytope Domains}},
  Numer. Math., 103 (2006), pp.~643--665.

\bibitem{pierce1989acoustics}
{\sc A.~Pierce}, {\em Acoustics: an introduction to its physical principles and
  applications}, J. Acoust. Soc. Am., 1989.

\bibitem{voronovich1999wave}
{\sc A.~Voronovich}, {\em Wave scattering from rough surfaces}, Springer,
  Berlin, 1999.

\end{thebibliography}
\end{document}